\documentclass[12pt]{amsart}
\usepackage{latexsym, amsmath, amscd, amssymb, amsthm,longtable} 
\usepackage[english]{babel}
  
\newtheorem{te}{Theorem}
\newtheorem{ex}{Example}

\begin{document}

\title { \vspace{3.5cm} \textbf{The Double Star Sequences and the General Second Zagreb Index}} 

\author{\textbf{Leonid Bedratyuk}}
\date{}
\maketitle 

\begin{center}
{\it Khmelnytskiy national university, Instytutska, 11,  Khmelnitsky, 29016, Ukraine}
\end{center}

\begin{center}
{\tt leonid.uk@gmal.com}
\end{center}

\begin{center}
\textbf{Abstract}
\end{center}

{\small For a simple graph we introduce  notions of the double star sequence,   the double star frequently sequence and prove that these sequences   are inverses of each other. As a consequence, we express  the general  second Zagreb index in terms of the double star sequence. Also,   we calculate the ordinary generating function and a linear recurrence relation for the sequence of  the general  second Zagreb indexes.}

\renewcommand{\baselinestretch}{1.5}

\section{Introduction}

Let $G$ be a simple graph whose vertex and edge sets are $V (G)$ and $E(G)$, respectively. Let $d_v$ be the degree  of the vertex $v \in  V (G).$  For any real $p$ the general second Zagreb index is defined by  
$$
M_2^{(p)}(G)=\sum_{uv \in E(G)} (d_u \, d_v)^{p},
$$
see \cite{BE} for more details.
Put $n=|V(G)|$ and $m=|E(G)|$. The  double star graph $S_{a,b}, 0 \leq a\leq b $ is a three with the following degree sequence 
$$a+1, b+1, \underbrace{1,\ldots , 1}_{a+b \text{ times} }.$$
For example  $S_{0,0}=P_2,$  $S_{0,1}=P_3$, $S_{1,1}=P_4$ and $S_{0,a}$ is the star  $S_{a+1}.$

Denote by $S_{a,b}(G)$ the number of subgraphs  of $G$  which  are isomorphic to the double star  $S_{a,b}.$ 
For instance $S_{0,0}(G)$ is equal to the number  of edges of  $G.$  
It is easy to see that there exists only one connected graph $G$ of order  $n$ such that   $S_{n-2,n-2}(G)=1$,  namely, the graph with the degree sequence 
$$n-1, n-1, \underbrace{2,\ldots , 2}_{n-2 \text{ times} }.$$
Also, we have that $S_{a,n-1}(G)=0$  for all $a.$ 

The triangle
\begin{align*}
&S_{0,0}(G), S_{0,1}(G), \ldots, S_{0,n-2}(G),\\
& S_{1,2}(G), \ldots, S_{1,n-2}(G),\\
&\ldots \ldots \ldots \ldots\\
&S_{n-2,n-2}(G),
\end{align*}
is called  the  \textit{double star sequence} of a graph $G.$

The value $S_{a,b}(G)$ can be counted   by  simple combinatorial techniques:

$$
S_{a,b}(G)=\begin{cases} \displaystyle \sum_{\{u,v\}  \in E(G)} \left( \binom{d_u-1}{a}\binom{d_v-1}{b}+ \binom{d_u-1}{b}\binom{d_v-1}{a} \right), a \le b,\\
 \displaystyle \sum_{\{u,v\}  \in E(G)}  \binom{d_u-1}{a}\binom{d_v-1}{a}, a = b, \end{cases}
$$
 see   \cite{SG}, the formulas (1.7)-(1.8).

For small $p$ we have 
$$
M_2^{(0)}(G)=S_{0,0}(G)=m,
$$
and 
\begin{gather*}
M_2^{(1)}(G)=M_2(G)=\sum_{uv \in E(G)} d_u d_v= \sum_{uv \in E(G)} 1+(d_u-1)+(d_v-1)+(d_u-1)(d_v-1)
 \\=m+\sum_{uv \in E(G)} \binom{d_u-1}{1}\binom{d_v-1}{0}+\binom{d_u-1}{0}\binom{d_v-1}{1}+\binom{d_u-1}{1}\binom{d_v-1}{1}\\= S_{0,0}(G)+S_{0,1}(G)+S_{1,1}(G).
\end{gather*}

Also,   by using the simple identity
\begin{gather*}
((x+1)(y+1))^2
=1+3 \left(\binom{x}{1}+\binom{y}{1}\right)+2 \left(\binom{x}{2}+\binom{y}{2}\right)  +9\binom{x}{1}\binom{y}{1}\\+6 \left(\binom{x}{1}\binom{y}{2}+\binom{x}{2}\binom{y}{1}\right)
+4\binom{x}{2}\binom{y}{2},
\end{gather*}
we have that 
\begin{gather*}
M_2^{(2)}(G)=\sum_{uv \in E(G)} (d_u d_v)^2=S_{{0,0}}(G)+3\,S_{{0,1}}(G)+2\,S_{{0,2}}(G)+9\,S_{{1,1}}(G)\\+6\,S_{{1,2}}(G)+4\,S_{{2
,2}}(G)
.
\end{gather*}

We  generalize these  expressions for the general second  Zagreb index $M_2^{(p)}(G)$  for any natural  $p.$ See also \cite{BS} for similar results for the general first  Zagreb index. 
The main result of the  paper is the formula for expressing  of the general second  Zagreb index in terms of the double star sequence:

\begin{gather*}
M_2^{(p)}(G)=\sum_{i=0}^{n-2} \sum_{k=i}^{n-2} i! k! \left\{ {p+1 \atop i+1}\right\} \left\{ {p+1 \atop k+1}\right\}S_{i,k}(G), p \in \mathbb{N},
\end{gather*}
here  $\displaystyle \left\{ {p \atop i}\right\}$ are the Stirling numbers of the second kind.

Also, we calculate the ordinary generating function for the integer sequence $\{M_2^{(p)}(G) \}$. Denote by $C_{n}$  the set 
$$
C_{n}=\{ i \cdot j \mid 0 \leq  i,j \leq n\}.
$$
Then
$$
\sum_{p=0}^\infty M_2^{(p)}(G) z^p =\frac{\displaystyle \sum_{k=0}^{|C_{n-1}|-1} \left( \sum_{i=\max(0, k-|C_{n-1}|)}^{\min(k,|C_{n-1}|)} \left [ {C_{n-1}} \atop {|C_{n-1}|-i}\right ] M_2^{(k-i)}(G) \right) t^k }{\prod\limits_{c \in C_{n-1}}(1- c \cdot t)},
$$

and the linear  recurrence relation for the integer sequence $\{ M_2^{(p)}(G) \}:$  

$$
M_2^{(|C_{n-1}|)}(G)=-\sum_{i=1}^{|C_{n-1}|-1} \left [ {C_{n-1}} \atop {i}\right ] M_2^{(i)}(G),
$$

here 
 $\displaystyle \left[ {C_{n-1} \atop i}\right]$ are the Comtet numbers of the first kind associated with the set $C_{n-1}.$


\section{Double star and double frequently sequences}


Any edge   $\{ u,v \} \in E(G)$ is the double star  $S_{d_u-1, d_v-1}, d_u \leq d_v.$
Let $f_{i,j}$ denotes the number of edges in $G$  which are the double stars $S_{i,j}$. For example,  $f_{0,0}$ is the number of isolated edges in $G.$  For the graph $G$ with the degree sequence 
$$n-1, n-1, \underbrace{2,\ldots , 2}_{n-2 \text{ times} },$$
we have $f_{n-2,n-2}=1,$ $f_{0,n-2}=2(n-2)$ and $f_{i,j}=0$.
The integer triangle
\begin{align*}
&f_{0,0}, f_{0,1}, \ldots, f_{0,n-2},\\
&f_{1,1}, \ldots, f_{1,n-2},\\
&\ldots \ldots \ldots \ldots\\
&f_{n-2,n-2},
\end{align*}
is called  the  \textit{double star frequently sequence} of a graph $G.$

\begin{ex}{\rm  Let $G=K_4.$   Then the  double star  sequence  and the  double star frequently sequence  have  form 
$$
\begin{array}{ll}
S_{0,0}(G)=6, S_{0,1}(G)=24,  S_{0,2}(G)=12, &f_{0,0}=0, f_{0,1}=0,  f_{0,2}=0,\\
S_{1,1}(G)=24, S_{1,2}(G)=24,&f_{1,1}=0,  f_{1,2}=0,\\
S_{2,2}(G)=6,&f_{2,2}=6.
\end{array}
$$
}
\end{ex}

The double star sequence  and the double star frequency  sequence  are  a pair of multinomial inverse  sequences, see   \cite[section 3.5]{Rio}  for more details.
The following theorem holds.

\begin{te} Let $G$ be a simple graph. Then 

\begin{align*}
(i) & S_{a,b}(G)=\begin{cases}\displaystyle \sum_{i=0}^{n-2} \sum_{j=i}^{n-2} \left( \binom{i}{a}\binom{j}{b}+ \binom{i}{b}\binom{j}{a} \right) f_{i,j}, a<b,\\ \displaystyle \sum_{i=0}^{n-2} \sum_{j=i}^{n-2}  \binom{i}{a}\binom{j}{a} f_{i,j}, a=b,  \end{cases}\\
(ii) & f_{a,b}=\begin{cases}\displaystyle \sum_{i=0}^{n-2} \sum_{j=i}^{n-2} (-1)^{a+b+i+j} \left( \binom{i}{a}\binom{j}{b}+ \binom{i}{b}\binom{j}{a} \right) S_{i,j}(G),a<b,\\ \displaystyle \sum_{i=0}^{n-2}\sum_{j=i}^{n-2} (-1)^{i+j} \binom{i}{a} \binom{j}{a}S_{i,j}(G), a=b. \end{cases}
\end{align*}

\end{te}

\begin{proof}  $(i)$ The statement follows immediately from the definitions of $f_{i,j}.$

$(ii)$  For simplicity, we consider only the case $a=b$.  We have

\begin{gather*}
f_{a,a}=\sum_{i=0}^{n-2}\sum_{j=i}^{n-2} (-1)^{i+j} \binom{i}{a} \binom{j}{a}S_{i,j}(G)=\sum_{i=0}^{n-2} (-1)^{i+i} \binom{i}{a} \binom{i}{a}S_{i,i}(G)\\+\sum_{i=0}^{n-2}\sum_{j=i+1}^{n-2} (-1)^{i+j} \binom{i}{a} \binom{j}{a}S_{i,j}(G)=\sum_{i=0}^{n-2}  \binom{i}{a} \binom{i}{a}\sum_{p=0}^{n-2} \sum_{q=p}^{n-2}  \binom{p}{i}\binom{q}{i} f_{p,q}\\+\sum_{i=0}^{n-2}\sum_{j=i+1}^{n-2} (-1)^{i+j} \binom{i}{a} \binom{j}{a}\sum_{p=0}^{n-2} \sum_{q=p}^{n-2} \left( \binom{p}{i}\binom{q}{j}+ \binom{p}{j}\binom{q}{i} \right) f_{p,q}\\=
\sum_{p=0}^{n-2} \sum_{q=p}^{n-2}\sum_{i=0}^{n-2} \left(  \binom{i}{a} \binom{i}{a} \binom{p}{i}\binom{q}{i}+\sum_{j=i+1}^{n-2} (-1)^{i+j}\binom{i}{a} \binom{j}{a} \left( \binom{p}{i}\binom{q}{j}+ \binom{p}{j}\binom{q}{i} \right)\right) f_{p,q}.
\end{gather*}

Now we use the  orthogonal relation, see \cite{Rio}, 
$$
\sum_{i=0}^n (-1)^i \binom{i}{a} \binom{b}{i}=(-1)^a \delta_{a,b},
$$
and the simple  identity
\begin{gather*}
\left ( \sum_{i=0}^p a_i \right) \times \left ( \sum_{k=0}^p b_k \right)=\sum_{i=0}^n a_i b_i+ \sum_{i=0}^p \sum_{k=i+1}^p (a_i b_k+b_i a_k).
\end{gather*}

Then the coefficient of  $f_{p,q}$ equals 
\begin{gather*}
\sum_{i=0}^{n-2}  \binom{i}{a} \binom{i}{a} \binom{p}{i}\binom{q}{i}+\sum_{i=0}^{n-2}\sum_{j=i+1}^{n-2} (-1)^{i+j}\binom{i}{a} \binom{j}{a} \left( \binom{p}{i}\binom{q}{j}+ \binom{p}{j}\binom{q}{i} \right)\\=
\sum_{i=0}^{n-2}\sum_{j=0}^{n-2} (-1)^{i+j}\binom{i}{a} \binom{p}{i}\binom{j}{a} \binom{q}{j}\\=
\sum_{i=0}^{n-2} (-1)^i \binom{i}{a} \binom{p}{i} \times \sum_{j=0}^{n-2} (-1)^j \binom{j}{a} \binom{q}{j}=(-1)^a \delta_{p,a} (-1)^a \delta_{q,a}.
\end{gather*}
Thus 
\begin{gather*}
\sum_{i=0}^{n-2}\sum_{j=i}^{n-2} (-1)^{i+j} \binom{i}{a} \binom{j}{a}S_{i,j}(G)= (-1)^a \delta_{p,a} (-1)^a \delta_{q,a} f_{p,q}=f_{a,a},
\end{gather*}
as required.

The proof of the case $a \neq b$ is almost identical to that of the case $a=b$ and is omitted.
\end{proof}

As a consequence, we  obtain  the  double star variant of the Handshaking lemma

$$
\sum_{i=0}^{n-2} \sum_{j=i}^{n-2} f_{i,j}=\sum_{i=0}^{n-2} \sum_{j=i}^{n-2} \binom{i}{0} \binom{j}{0} f_{i,j}=S_{0,0}(G)=m.
$$


From \cite{DM} we know that

$$
\sum_{uv \in E(G)}\left( \frac{1}{d_u}+\frac{1}{d_v} \right)=n-n_0,
$$ 
here $n_0$ is the number of isolated vertices of $G.$ Then by definition of  the double star frequently sequence  we get  that
\begin{gather*}
\sum_{i=0}^{n-2}\sum_{j=i}^{n-2} \left(\frac{1}{i+1}+\frac{1}{j+1} \right)  f_{i,j}=n-n_0.
\end{gather*}


\section{The second general Zagreb index}


For the second general Zagreb index we have 
$$
 M_2^{(p)}(G)=\sum_{i=0}^{n-2}\sum_{j=i}^{n-2} ((i+1) (j+1))^p f_{i,j}.
$$

Now we can express the second general Zagreb index $M_2^p(G)$  in terms of double star sequence:
 
\begin{te} 
\begin{gather*}
M_2^p(G)=\sum_{i=0}^{n-2} \sum_{k=i}^{n-2} i! k! \left\{ {p+1 \atop i+1}\right\} \left\{ {p+1 \atop k+1}\right\}S_{i,k}(G),
\end{gather*}
for any natural number $p.$
\end{te}

\begin{proof}
We have 

\begin{gather*}
\sum_{i=0}^{n-2} \sum_{k=i}^{n-2} i! k! \left\{ {p+1 \atop i+1}\right\} \left\{ {p+1 \atop k+1}\right\}S_{i,k}(G)=\sum_{i=0}^{n-2} \sum_{k=i+1}^{n-2} i! k! \left\{ {p+1 \atop i+1}\right\} \left\{ {p+1 \atop k+1}\right\}S_{i,k}(G)\\+\sum_{i=0}^{n-2}  \left(i!  \left\{ {p+1 \atop i+1}\right\}\right)^2 S_{i,i}(G).
\end{gather*}
Simplify the second sum 
\begin{gather*}
\sum_{i=0}^{n-2}  \left(i!  \left\{ {p+1 \atop i+1}\right\}\right)^2 S_{i,i}(G)=\sum_{i=0}^{n-2}  \left(i!  \left\{ {p+1 \atop i+1}\right\}\right)^2 \sum_{s=i}^{n-2} \sum_{t=s}^{n-2}  \binom{s}{i}\binom{t}{i} f_{s,t}\\=
\sum_{s=0}^{n-2} \sum_{t=s}^{n-2} \left( \sum_{i=0}^{p}\left(i!  \left\{ {p+1 \atop i+1}\right\}\right)^2  \binom{s}{i}\binom{t}{i}\right) f_{s,t}. 
\end{gather*}

Simplify the first sum
\begin{gather*}
\sum_{i=0}^{n-2} \sum_{k=i+1}^{n-2} i! k! \left\{ {p+1 \atop i+1}\right\} \left\{ {p+1 \atop k+1}\right\}S_{i,k}(G)\\=\sum_{i=0}^{n-2} \sum_{k=i+1}^{n-2} i! k! \left\{ {p+1 \atop i}\right\} \left\{ {p+1 \atop k+1}\right\}\sum_{s=i}^{n-2} \sum_{t=s}^{n-2} \left( \binom{s}{i}\binom{t}{k}+ \binom{s}{k}\binom{t}{i} \right) f_{s,t}\\=
\sum_{s=0}^{n-2} \sum_{t=s}^{n-2}\left(  \sum_{i=0}^{p} \sum_{k=i+1}^{p} i! k! \left\{ {p+1 \atop i+1}\right\} \left\{ {p+1 \atop k+1}\right\}\left( \binom{s}{i}\binom{t}{k}+ \binom{s}{k}\binom{t}{i} \right)  \right) f_{s,t}.
\end{gather*}

Thus the coefficient of  $f_{s,t}$ in 
$$
\sum_{i=0}^{p} \sum_{k=i}^{p} i! k! \left\{ {p+1 \atop i+1}\right\} \left\{ {p+1 \atop k+1}\right\}S_{i,k}(G),
$$
equals
\begin{gather*}
\sum_{i=0}^{p}\left(i!  \left\{ {p+1 \atop i+1}\right\}\right)^2  \binom{s}{i}\binom{t}{i}+ \sum_{i=0}^{p} \sum_{k=i+1}^{p} i! k! \left\{ {p+1 \atop i+1}\right\} \left\{ {p+1 \atop k+1}\right\}\left( \binom{s}{i}\binom{t}{k}+ \binom{s}{k}\binom{t}{i} \right) \\=\sum_{i=0}^p\binom{s}{i} i! \left\{ {p+1 \atop i+1}\right\}  \times \sum_{k=0}^p\binom{t}{k} k! \left\{ {p+1 \atop k+1}\right\}=((s+1)(t+1))^p.
\end{gather*}
We are using here the identity
$$
\sum_{i=0}^p\binom{s}{i} i! \left\{ {p+1 \atop i+1}\right\} =(s+1)^p,
$$
which  can be derived from the formula 
\begin{gather*}
\left\{ {p+1 \atop i+1}\right\}=\sum_{j=i}^p \binom{p}{j}\left\{ {j \atop i}\right\},
\end{gather*}
see \cite{GKP}.
In fact, we have
\begin{gather*}
\sum_{i=0}^p\binom{s}{i} i! \left\{ {p+1 \atop i+1}\right\}=\sum_{i=0}^p\binom{s}{i} i! \sum_{j=i}^p \binom{p}{j}\left\{ {j \atop i}\right\}\\=\sum_{j=0}^p \binom{p}{j}\sum_{i=0}^j\binom{s}{i} i! \left\{ {j \atop i}\right\}=\sum_{j=0}^p \binom{p}{j} s^j=(s+1)^p.
\end{gather*}
Here, we used the following simple  summation interchange formula
$$
\sum_{i=0}^n a_i \sum_{j=i}^n b_{j} =\sum_{j=0}^n b_i \sum_{i=0}^j a_i.
$$

Thus  
\begin{gather*}
\sum_{i=0}^{n-2} \sum_{k=i}^{n-2} i! k! \left\{ {p+1 \atop i+1}\right\} \left\{ {p+1 \atop k+1}\right\}S_{i,k}(G)=\sum_{s=0}^{p} \sum_{t=s}^{p} ((s+1)\,(t+1))^p f_{s,t}=M_2^{(p)}(G),
\end{gather*}
as required.

\end{proof}

\begin{ex}{\rm For any simple graph $G$  we have
\begin{align*}
&M_2^{(0)}(G)=S_{0,0}(G),\\
&M_2^{(1)}(G)=S_{{0,0}}(G)+S_{{0,1}}(G)+S_{{1,1}}(G),\\
&M_2^{(2)}(G)=S_{{0,0}}(G)+3\,S_{{0,1}}(G)+2\,S_{{0,2}}(G)+9\,S_{{1,1}}(G)+6\,S_{{1,2}}(G)+4\,S_{{2
,2}}(G),
\\
&M_2^{(3)}(G)=S_{{0,0}}(G)+7\,S_{{0,1}}(G)+12\,S_{{0,2}}(G)+6\,S_{{0,3}}(G)+49\,S_{{1,1}}(G)+84\,S_
{{1,2}}(G)\\&42\,S_{{1,3}}(G)+144\,S_{{2,2}}(G)+72\,S_{{2,3}}(G)+36\,S_{{3,3}}(G)
.
\end{align*}
}
\end{ex}
Similar formulas  for  triangle free graphs are derived in \cite{SG}.


\section{Recurrence relations for the generalized second  Zagreb indexes}


The following generalization of the Srirling numbers of the first kind is well known, see 
 \cite{COM}, \cite{TM}. Let $S$ be an arbitrary set of natural numbers of cardinality  $|S|.$ Then \textit{the Comtet numbers  of the first kind} $\displaystyle \left [ {S \atop i}\right ]$ associated with  the set $S$ are defined by
$$
(z)_S=\prod_{s \in S}(z-s)=\sum_{i=0}^{|S|} \left [ {S \atop i}\right ] z^i,
$$
here $(z)_S$ is the generalized Pochhamer symbols.

The   Comtet numbers  of the first kind associated with  the set $S=\{ 0,1,\cdots, n-1 \}$  coincides with the usual Stirling numbers  of the first kind.

\begin{ex}{\rm Denote by $C_n$ the set 
$$
C_{n}= \{  i\cdot j \mid 0 \leq i,j \leq n \}.
$$
For  $n=4$ we have 
$$
C_{n}= \{0, 1, 2, 3, 4, 6, 8, 9, 12, 16 \},
$$
and 

\begin{gather*}
(z)_{C_4}=z \left( z-1 \right)  \left( z-2 \right)  \left( z-3 \right)  \left( z-4 \right)  \left( z-6 \right)  \left( z-8 \right)  \left( z-9 \right) 
 \left( z-12 \right)  \left( z-16 \right)\\={z}^{10}-61\,{z}^{9}+1555\,{z}^{8}-21655\,{z}^{7}+180628\,{z}^{6}-929908\,{z}^{5}+
2932320\,{z}^{4}-\\-5411520\,{z}^{3}+5239296\,{z}^{2}-1990656\,z.
\end{gather*}
Thus,  the corresponding Comtet numbers of the first kind are as follows:
\begin{gather*}
 \left [ {C_4 \atop 1}\right ]=-1990656,  \left [ {C_4 \atop 2}\right ]=5239296, \left [ {C_4 \atop 3}\right ]=-5411520,\left [ {C_4 \atop 4}\right ]=2932320,\left [ {C_4 \atop 5}\right ]=-929908,\\ \left [ {C_4 \atop 6}\right ]=180628,\left [ {C_4 \atop 7}\right ]=-21655,\left [ {C_4 \atop 8}\right ]=1555,\left [ {C_4 \atop 9}\right ]=-61,\left [ {C_4 \atop 10}\right ]=1, \left [ {C_4 \atop 0}\right ]=0.
\end{gather*}
}
\end{ex}

Let 

$$
\mathcal{G}(M_2,t)=\sum_{p=0}^\infty M_2^{(p)}(G) t^p,
$$

be the ordinary generating functions of the sequence of the second  general  Zagreb indexes. Let us express the generating functions $\mathcal{G}(M_2,t)$  in terms of the   double star sequences. The following theorem holds.

\begin{te} Let $\mathcal{G}(M_2,t)$   be the  ordinary generating functions of the sequence of the second  general  Zagreb indexes. Then  

\begin{align*}
(i) & \,\mathcal{G}(M_2,t)=
\frac{\displaystyle \sum_{k=0}^{|C_{n-1}|-1} \left( \sum_{i=\max(0, k-|C_{n-1}|)}^{\min(k,|C_{n-1}|)} \left [ {C_{n-1}} \atop {|C_{n-1}|-i}\right ] M_2^{(k-i)}(G) \right) t^k }{\prod\limits_{c \in C_{n-1}}(1- c t)},\\
(ii) & M_2^{(|C_{n-1}|)}(G)=-\sum_{i=1}^{|C_{n-1}|-1} \left [ {C_{n-1}} \atop {i}\right ] M_2^{(i)}(G), 
\end{align*}

 where $\displaystyle \left [ {S \atop i}\right ]$  is the Comtet numbers of the first kind associated with the  set $S.$

\end{te}

\begin{proof}

We have
\begin{gather*}
\mathcal{G}(M_2,t)=\sum_{p=0}^\infty M_2^{(p)}(G) t^p=\sum_{p=0}^\infty\left(\sum_{i=1}^{n-2} \sum_{k=i}^{n-2} i! k! \left\{ {p+1 \atop i+1}\right\} \left\{ {p+1 \atop k+1}\right\}S_{i,k}(G)\right) t^p\\=\sum_{i=1}^{n-1} \sum_{k=i}^{n-1} \left(  \sum_{p=0}^\infty  i! k!\left\{ {p+1 \atop i+1}\right\} \left\{ {p +1\atop k+1}\right\} t^p \right) S_{i,k}(G)\\=
\sum_{i=1}^{n-1} \sum_{k=i}^{n-1} \left(  \sum_{p=i}^\infty  i! k!\left\{ {p+1 \atop i+1}\right\} \left\{ {p+1 \atop k+1}\right\} t^p \right) S_{i,k}(G).
\end{gather*}
Since 
$$
{n \brace k}=\frac{1}{k!} \sum_{i=0}^k (-1)^{k-i} \binom{k}{i} i^n,
$$
we get 
\begin{gather*}
\sum_{p=i}^\infty i! k! \left\{ {p +1\atop i+1}\right\} \left\{ {p+1 \atop k+1}\right\} t^p\\=\sum_{p=i}^\infty \left(  \sum_{r=0}^{i+1} \frac{(-1)^{i+1-r}}{i+1} \binom{i+1}{r} r^{p+1} \sum_{j=0}^{k+1} \frac{ (-1)^{k+1-j}}{k+1} \binom{k+1}{j} j^{p+1} \right) t^p\\=
 \frac{1}{(i+1)(k+1)}\left(  \sum_{r=0}^{i+1} (-1)^{i+1-r} r \binom{i+1}{r}  \sum_{j=0}^{k+1} (-1)^{k+1-j} j\,\binom{k+1}{j}  \right)\sum_{p=i}^\infty (r j t)^p\\=
  \frac{1}{(i+1)(k+1)}\sum_{r=0}^{i+1} \sum_{j=0}^{k+1} (-1)^{i+1-r} (-1)^{k+1-j}  rj \binom{i+1}{r}    \binom{k+1}{j}  \frac{(rjt)^i}{1-(rj)t}.
\end{gather*}
It follows that  the partial fraction  decomposition of the generating function $\mathcal{G}(M_2,t)$ is a linear combination of fractions of the form
$$
\frac{1}{1-(rj)t}, 0 \leq r,j \leq n-1.
$$
After simplification we get 
\begin{gather*}
\mathcal{G}(M_2,t)= \frac{a_0+a_1 t+a_2 t^2+\cdots+a_{|C_{n-1}|-1} t^{|C_{n-1}|-1}}{\prod\limits_{c \in C_{n-1}}(1- c t)}
\end{gather*}
for some unknown numbers $a_0, a_1, \ldots, a_{n-1}.$  To define these numbers let us observe that 
\begin{gather*}
\prod\limits_{c \in C_{n-1}}(1- c \,t)=t^{|C_{n-1}|}\prod\limits_{c \in C_{n-1}}\left(\frac{1}{t}- c\right) =t^{|C_{n-1}|}\left(\frac{1}{t} \right)_{C_{n-1}}=\sum_{i=0}^{C_{n-1}}  \left [ {C_{n-1} \atop i}\right ] t^{|C_{n-1}|-i}.
\end{gather*}
Now 
\begin{gather*}
a_0+a_1 t+a_2 t^2+\cdots+a_{|C_{n-1}|-1} t^{|C_{n-1}|-1}=\left(\sum_{i=0}^{|C_{n-1}|}  \left [ {C_{n-1} \atop i}\right ] t^{|C_{n-1}|-i}\right) \left( \sum_{p=0}^\infty M_2^{(p)}(G) t^p \right)\\
=\sum_{p=0}^\infty  \left( \sum_{i=\max(0, p-|C_{n-1}|)}^{\min(p,|C_{n-1}|)} \left [ {C_{n-1}} \atop {|C_{n-1}|-i}\right ] M_2^{(p-i)}(G)  \right) t^p.
\end{gather*}
Equating coefficients of $t^p$  yelds 
$$
a_p= \sum_{i=\max(0, p-|C_{n-1}|)}^{\min(p,|C_{n-1}|)} \left [ {C_{n-1}} \atop {|C_{n-1}|-i}\right ] M_2^{(p-i)}(G).
$$
Therefore
\begin{gather*}
\mathcal{G}(M_2,t)=
\frac{\displaystyle \sum_{p=0}^{|C_{n-1}|-1} \left( \sum_{i=\max(0, p-|C_{n-1}|)}^{\min(p,|C_{n-1}|)} \left [ {C_{n-1}} \atop {|C_{n-1}|-i}\right ] M_2^{(p-i)}(G) \right) t^p }{\prod\limits_{c \in C_{n-1}}(1- c t)}.
\end{gather*}

$(ii)$ Since $a_{p}=0$ for $p \geq |C_{n-1}|$  we have the identity 
$$
a_{|C_{n-1}|}=\sum_{i=\max(0, p-|C_{n-1}|)}^{\min(p,|C_{n-1}|)} \left [ {C_{n-1}} \atop {|C_{n-1}|-i}\right ] M_2^{(k-i)}(G)=0,
$$
or 
$$
\sum_{i=0}^{|C_{n-1}|} \left [ {C_{n-1}} \atop {|C_{n-1}|-i}\right ] M_2^{(|C_{n-1}|-i)}(G)=\sum_{i=0}^{|C_{n-1}|} \left [ {C_{n-1}} \atop {i}\right ] M_2^{(i)}(G)=0
$$
Tacking into account 
$$
\left [ {C_{n-1}} \atop {0}\right ]=0, \left [ {C_{n-1}} \atop {|C_{n-1}|}\right ]=1
$$
we can rewrite the last expression the form
$$
M_2^{(|C_{n-1}|)}(G)=-\sum_{i=1}^{|C_{n-1}|-1} \left [ {C_{n-1}} \atop {i}\right ] M_2^{(i)}(G).
$$

\end{proof}

\begin{ex}{\rm For   $n=4$ we have
$$
C_3=\{ {0, 1, 2, 3, 4, 6, 9}\}.
$$ 
Then 
\begin{gather*}
(z)_{C_3}= z \left( z-1 \right)  \left( z-2 \right)  \left( z-3 \right)  \left( z
-4 \right)  \left( z-6 \right)  \left( z-9 \right) \\={z}^{7}-25\,{z}^{6
}+239\,{z}^{5}-1115\,{z}^{4}+2664\,{z}^{3}-3060\,{z}^{2}+1296\,z.
\end{gather*}
The Comtet numbers of the first kind are as follows: 
\begin{align*}
 &\left [ {C_3 \atop 0}\right ]=0, \left [ {C_3 \atop 1}\right ]=1296,  \left [ {C_3 \atop 2}\right ]=-3060,  \left [ {C_3 \atop 3}\right ]=2664, \\ &\left [ {C_3 \atop 4}\right ]=-1115,   \left [ {C_3 \atop 5}\right ]=239,  \left [ {C_3 \atop 6}\right ]=-25, \left [ {C_3 \atop 7}\right ]= 1.
\end{align*}
Then for a simple graph $G$ with $4$ vertices the following recurrence relation holds
\begin{gather*}
 M_2^{(7)}(G)= -1296\, M_2^{(1)}(G)+3060\,M_2^{(2)}(G)-2664 M_2^{(3)}(G)+1115 M_2^{(4)}(G)-\\-239 M_2^{(5)}(G)+25 M_2^{(6)}(G).
\end{gather*}
}
\end{ex}


\begin{thebibliography}{30}



\bibitem{BE}

B. Bollobas, P. Erd\"os, Graphs of extremal weights, \textit{Ars Comb.} \textbf{50} (1998) 225--233.




\bibitem{SG}
G. Xavier, E. Suresh, I. Gutman, Counting relations for general Zagreb
indices, \textit{Kragujevac Journal of Mathematics}, \textbf{38(1)} (2014), 95--103.

\bibitem{BS}

L. Bedratyuk, O. Savenko,   The Star Sequence and the General First Zagreb Index,  \textit{MATCH Communications in Mathematical and in Computer Chemistry,} \textbf{79(2)} (2018) 407--414.


\bibitem{Rio}

J. Riordan,  \textit{Combinatorial Identities,} New York, Wiley,  1979.

\bibitem{DM}

T. Doslic, B. Furtula, A. Graovac, I. Gutman, S. Moradi, Z. Yarahmadi, On Vertex-Degree-Based Molecular Structure Descriptors,
\textit{MATCH Commun. Math. Comput. Chem.\/} \textbf{66} (2011) 613--626.
	
		\bibitem{COM}
		L. Comtet, Nombres de Stirling generaux et fonctions symetriques, \textit{C. R. Acad. Sci. Paris. Ser. A} \textbf{275} (1972), 747–-750.
		
		\bibitem{TM}
T. Mansour, M. Schork, \textit{Commutation Relations, Normal Ordering, and Stirling Numbers,} CRC Press, 2015.


\bibitem{GKP}

R.  Graham, D. Knuth,  O. Patashnik, \textit{Concrete Mathematics}, Addison-Wesley, Reading, 1989.



\end{thebibliography}
\end{document}